\documentclass[a4paper,12pt,oneside]{article}
\usepackage{amsmath, amssymb, graphicx, amsthm, mathrsfs}
\usepackage{mathptmx}
\usepackage{mathtools} 
\usepackage{geometry}
\usepackage{enumitem}  
\usepackage{hyperref}
\newtheorem{theorem}{Theorem}[section]
\newtheorem{definition}[theorem]{Definition}
\newtheorem{lemma}[theorem]{Lemma}

\newtheorem{corollary}[theorem]{Corollary}

\newtheorem{remark}[theorem]{Remark}
\newtheorem{example}[theorem]{Example}
\numberwithin{equation}{section}

\title{Weakly Compatible Mappings and Common Fixed Points Under Generalized Contractive Conditions}

\author{Alemayehu Negash$^{*1}$ and Meaza Bogale$^2$\\
$^{1,2}$ Department of Mathematics, Hampton University,  USA  \\
$^*$ \textit{Corresponding author.} \\
\textit{Email address}: alemayehu.negash@hamptonu.edu,\\  meaza.fantahun@hamptonu.edu.}
\date{}

\begin{document}

\maketitle

\begin{abstract}
  This paper establishes new common fixed point theorems for weakly compatible mappings in metric spaces, relaxing traditional requirements such as continuity, compatibility, and reciprocal continuity. We present a unified framework for three self-mappings \(T\), \(f\), and \(g\) with a contractive condition involving a control function \(\psi\), along with corollaries extending results to pairs of mappings and upper semi-continuous control functions. Further generalizations include iterated mappings and sequences of mappings. Rigorous examples demonstrate the necessity of hypotheses and show our results strictly generalize theorems by Al-Thagafi \emph{et. al.} \cite{Al-Thagafi2006}, Babu \emph{et. al.} \cite{Babu2007}, Jungck \cite{Jungck1976,Jungck1986}, Singh \cite{Singh1986,Singh1997a}, Som \cite{Som2003}, Song \cite{Song2007} and Zhang \emph{et. al.} \cite{Zhang2008}. Key advancements include eliminating completeness assumptions on the entire space and relaxing mapping compatibility conditions. 


\textbf{AMS(2000) Mathematics Subject Classification:} 47H10, 54H25.
\end{abstract}

\section{Introduction}
Fixed point theory has been extensively developed since Banach's contraction principle. This paper presents significant generalizations by:
\begin{itemize}
    \item Establishing common fixed points under \textit{weak compatibility} without continuity/compatibility
    \item Using control functions (\(\psi\), \(\varphi\)) in contractive inequalities
    \item Requiring completeness only on subspaces (\(T(K)\), \(f(K)\), or \(g(K)\))
    \item Extending results to sequences and iterated mappings
\end{itemize}
Our results unify and improve theorems by Al-Thagafi \textit{et al.} \cite{Al-Thagafi2006}, Babu \textit{et al.} \cite{Babu2007}, Jungck \cite{Jungck1976}, Singh \cite{Singh1986,Singh1997a}, Som \cite{Som2003}, Song \cite{Song2007}, and Zhang \textit{et al.} \cite{Zhang2008}.

\begin{theorem}[Jungck \cite{Jungck1976}] \label{thrm 1.1}
Let $f$ be a continuous mapping of a complete metric space $(X,d)$ into itself. Then $f$ has a fixed point in $X$ if and only if there exists $r\in [0,1)$ and a mapping $T:X\to X$ which commutes with $f$ ($fT = Tf$) and satisfies:
\begin{equation} \label{(1.1.1)}
    T(X)\subseteq f(X),
\end{equation}
\begin{equation} \label{(1.1.2)}
    d(Tx,Ty)\leq r d(fx,fy)
\end{equation}
for all $x, y\in X$.

In fact, $f$ and $T$ have a unique common fixed point in $X$.
\end{theorem}

If $T$ and $f$ satisfy the inequality \eqref{(1.1.2)} then we say that $T$ is an $f$-contraction.

Singh \cite{Singh1997a} proved the following common fixed point theorem for a pair of continuous and commuting selfmappings.

\begin{theorem}[Singh \cite{Singh1997a}] \label{thrm 1.2}
Let two continuous and commuting mappings $T$ and $f$ from a complete metric space $(X,d)$ into itself satisfy the  conditions \eqref{(1.1.1)}
\begin{equation} \label{(1.2.1)}
    d(Tx,Ty)\leq a[d(Tx,fx)+d(Ty,fy)] + b[d(Tx,fy)+d(Ty,fx)] + c\,d(fx,fy)
\end{equation}
for all $x,y\in X$, where $a\geq 0$, $b\geq 0$, $c\geq 0$ satisfying $0<2a+2b+c<1$.

Then $f$ and $T$ have a unique common fixed point in $X$.
\end{theorem}

\begin{definition}[Jungck \cite{Jungck1986}]
Two mappings $f$ and $T$ of a metric space $(X,d)$ are called \textit{compatible} if, $\lim_{n\to\infty}d(fTx_n,Tfx_n)=0$ whenever $\{x_n\}$ is a sequence in $X$ such that $\lim_{n\to\infty}fx_n=\lim_{n\to\infty}Tx_n=t$ for some $t$ in $X$.
\end{definition}

\begin{definition}[Pant \cite{Pant1998}]
Two mappings $f$ and $T$ of a metric space $(X,d)$ are called \textit{reciprocal continuous} on $X$ if, $\lim_{n\to\infty}fTx_n=ft$ and $\lim_{n\to\infty}Tfx_n=Tt$ whenever $\{x_n\}$ is a sequence in $X$ such that $\lim_{n\to\infty}fx_n=\lim_{n\to\infty}Tx_n=t$ for some $t$ in $X$.
\end{definition}

We observe that if $f$ and $T$ are continuous then they are reciprocal continuous. But its converse need not be true \cite{Pant1998}.

Babu \emph{et al.} \cite{Babu2007} generalized Singh's Theorem \cite{Singh1997a} as follows:

\begin{theorem}[Babu \emph{et al.} \cite{Babu2007}, Corollary 4.2] \label{thrm 1.5}
Let $T$ and $f$ be selfmaps of a complete metric space $(X,d)$ into itself which satisfy $T(X)\subseteq f(X)$ and the pair $(T,f)$ be compatible. Assume that there is an $r\in[0,1)$ such that
\begin{equation}
d(Tx,Ty)\leq r\max\left\{d(fx,fy),d(Tx,fx),d(Ty,fy),\frac{1}{2}[d(Tx,fy)+d(Ty,fx)]\right\}
\end{equation}
for all $x,y\in X$. If either $T$ or $f$ is continuous, then for each $x_0\in X$, the sequence defined by $y_n=Tx_n=fx_{n+1}$, $n=0,1,2,...$ is Cauchy in $X$, and $\lim_{n\to\infty}y_n=z$ (say), $z$ in $X$ and $z$ is the unique common fixed point of $T$ and $f$.
\end{theorem}

Som \cite{Som2003} extended Singh's theorem \cite{Singh1997a} to a sequence of selfmaps in the
following way:

\begin{theorem}[Som \cite{Som2003}, Theorem 3, p. 311] \label{Thrm 1.6}
Let $(X,d)$ be a complete metric space. Let $\{T_{n}\}$ be a family of selfmaps of $X$ and $f$ be a continuous selfmap of $X$ such that $(T_{n},f)$ is a compatible pair and $T_{n}(X)\subseteq f(X)$ for each $n$, and for all $x,y$ in $X$,
\begin{equation} \label{(1.6.1)}
d(T_{i}x,T_{j}y) \leq a_{1}d(T_{i}x,fx) + a_{2}d(T_{j}y,fy) + a_{3}d(T_{i}x,fy) + a_{4}d(T_{j}y,fx) + a_{5}d(fx,fy)
\end{equation}
for all $x,y\in X$, where $a_{k}\geq 0$ for $k=1,2,3,4,5$ and $a_{1}+a_{2}+a_{3}+a_{4}+a_{5}<1$. Then $\{T_{n}\}$ and $f$ have a unique common fixed point in $X$.
\end{theorem}

From Theorem 2.1 of Babu \emph{et al.} \cite{Babu2004}, the following theorem is immediate by choosing $T$ in place of $A$ and $B$; and $f$ in place of $S$ and $T$; $c_{2}=0$ and $c_{3}=0$.

\begin{theorem}[Babu \emph{et al.} \cite{Babu2004}] \label{thrm 1.7}
Let $(X,d)$ be a metric space. Let $T$ and $f$ be selfmaps of a metric space $(X,d)$ satisfying
\begin{equation} \label{(1.7.1)}
T(X)\subseteq f(X)
\end{equation}
Assume that there exists a $c_{1}\in[0,1)$ such that
\begin{equation} \label{(1.7.2)}
d(Tx,Ty) \leq c_{1}\max\{d(Tx,fy),d(Ty,fx),d(fx,fy)\}
\end{equation}
for all $x,y\in X$. Further assume that the pair $(T,f)$ is weakly compatible, and either $T(X)$ or $f(X)$ is complete. Then $T$ and $f$ have a unique common fixed point in $X$.
\end{theorem}

In Theorem \eqref{Thrm 3.1} of Babu \emph{et al.} \cite{Babu2007}, if we choose $\psi$ as the identity map on $[0,\infty)$ then we have the following:

\begin{theorem}[Babu \emph{et al.} \cite{Babu2007}, Theorem 3.1] \label{thrm 1.8}
Let $T$ and $f$ be selfmaps of a complete metric space $(X,d)$ satisfying $T(X)\subset f(X)$ and assume that there is an $r\in[0,1)$ such that
\begin{equation}
d(Tx,Ty) \leq r \max\left\{d(fx,fy),d(Tx,fx),d(Ty,fy),\tfrac{1}{2}[d(Tx,fy)+d(Ty,fx)]\right\}
\end{equation}
for all $x,y\in X$; the pair $(T,f)$ is compatible and either $T$ or $f$ is continuous, or the pair $(T,f)$ is reciprocal continuous. Then $F(T)\cap F(f)$ is a singleton.
\end{theorem}

The purpose of this paper is to establish the existence of common fixed points for weakly compatible mappings. We extend this result to a sequence of self mappings. In these results we do not assume the properties like continuity, reciprocal continuity and compatibility of maps so that our result generalize and improve various other similar results in the literature, e.g., Al-Thagafi \emph{et al.} \cite{Al-Thagafi2006}, Babu \emph{et al.} \cite{Babu2007}, Jungck \cite{Jungck1976,Jungck1986}, Singh \cite{Singh1986,Singh1997a}, and Som \cite{Som2003}.

\section{Preliminaries}

Let $K$ be a nonempty subset of a metric space $(X,d)$ and $f$, $g$ and $T$ be three selfmaps of $K$, and $F(T)$ be the set of fixed points of $T$, i.e., $F(T) = \{x \in K : Tx = x\}$.

We denote the closure of a set $K$ by $\overline{K}$, the set of all positive integers by $\mathbb{N}$ and $\mathbb{R}_{+} = [0, \infty)$.

\begin{definition}[Jungck \emph{et al.} \cite{Jungck1998}]
Two mappings $f$ and $T$ of a metric space $(X,d)$ are called \textit{weakly compatible} if they commute at their coincidence points. That is, for $x$ in $X$, if $fx = Tx$ then $fTx = Tfx$.
\end{definition}

We observe that every compatible map is weakly compatible but every weakly compatible map need not be compatible \cite{Singh1997b}.

\begin{definition}[Zhang \emph{et al.} \cite{Zhang2008}]
A selfmap $T$ of $K$ is called $(f,g)$-\textit{Boyd-Wong contraction} if there exists $\varphi:\mathbb{R}_{+} \to \mathbb{R}_{+}$ which is upper semicontinuous and $0 \leq \varphi(t) < t$ for $t > 0$ such that
\begin{equation} \label{(2.2.1)}
d(Tx,Ty) \leq \varphi(d(fx,gy)), \quad \text{for all } x,y \in K.
\end{equation}
When $\varphi(t) = \beta t$ for some $t \in [0,1)$, then $T$ is called $(f,g)$-\textit{contraction}. If, in addition, $f = g$, $T$ is called $f$-\textit{contraction}.
\end{definition}

Al-Thagafi \emph{et al.} \cite{Al-Thagafi2006} proved the following generalization of Jungck's common fixed point theorem \cite{Jungck1986}.

\begin{theorem}[Al-Thagafi \emph{et al.} \cite{Al-Thagafi2006}, Theorem 2.1] \label{thrm 2.3}
Let $K$ be a subset of a metric space $(X,d)$, $f$ and $T$ selfmaps of $K$, and $T(K) \subseteq f(K)$. Suppose that $f$ and $T$ are weakly compatible, $T$ is $f$-contraction, and $T(K)$ is complete. Then $F(T) \cap F(f)$ is singleton.
\end{theorem}

In 2006, Zhang \emph{et al.} \cite{Zhang2008} established the following common fixed point theorem using $(f,g)$-Boyd-Wong contraction.

\begin{theorem}[Zhang \emph{et al.} \cite{Zhang2008}, Theorem 2.1] \label{thrm 2.4}
Let $K$ be a subset of a metric space $(X,d)$ and $T$, $f$, $g:K \to K$ three mappings which satisfies $\overline{T(K)} \subset f(K) \cap g(K)$, and $T$ be a $(f,g)$-Boyd-Wong contraction. Suppose that either $\overline{T(K)}$ or $f(X)$ or $g(K)$ is complete. Then
\begin{enumerate}
    \item[(i)] there exists $z,u,v \in K$ such that $fu = Tu = z = Tv = gv$.
\end{enumerate}
If, in addition, $(T,f)$ and $(T,g)$ are weakly compatible, then
\begin{enumerate}
    \item[(ii)] $F(T) \cap F(f) \cap F(g)$ is singleton.
\end{enumerate}
\end{theorem}
\begin{definition}[Song \cite{Song2007}]
A mapping $T:K \to K$ is called a \textit{generalized $(f,g)$-contraction} if there exists an $r \in [0,1)$ such that
\begin{equation} \label{(2.5.1)}
d(Tx,Ty) \leq r \max\left\{d(Tx,fx), d(Ty,gy), \tfrac{1}{2}[d(Tx,gy) + d(Ty,fx)], d(fx,gy)\right\}
\end{equation}
for all $x,y \in X$.
\end{definition}

Song \cite{Song2007} established the following common fixed point theorem for three selfmaps $T$, $f$ and $g$ where $T$ is an $(f,g)$-contraction.

\begin{theorem}[Song \cite{Song2007}] \label{thrm 2.6}
Let $K$ be a subset of a metric space $(X,d)$ and $T, f, g:K \to K$ three selfmaps which satisfies $\overline{T(K)} \subset f(K) \cap g(K)$. Suppose that $\overline{T(K)}$ is complete, and $T$ is a generalized $(f,g)$-contraction with constant $r \in [0,1)$. Then neither $C(T,f)$ nor $C(T,g)$ is empty, where $C(T,f)$ is the set of coincidence points of $T$ and $f$; and $C(T,g)$ is the set of coincidence points of $T$ and $g$.

Moreover, if, in addition, both $(T,f)$ and $(T,g)$ are weakly compatible, then $F(T) \cap F(f) \cap F(g)$ is singleton.
\end{theorem}

To prove Theorem \eqref{thrm 2.4} and Theorem \eqref{thrm 2.6}, the respective authors chose the sequence $\{x_n\}$ as follows: Take $x_0 \in K$. As $\overline{T(K)} \subset f(K) \cap g(K)$, we choose a sequence $\{x_n\}$ in $K$ such that
\begin{equation}
(A) \quad Tx_{2n} = fx_{2n+1} \text{ and } Tx_{2n+1} = gx_{2n+2} \text{ for all } n \geq 0.
\end{equation}

\begin{remark}
Zhang \emph{et al.} \cite{Zhang2008} and Song \cite{Song2007} obtained the following inequalities using the sequence (A):
\begin{equation} \label{(2.7.1)}
  d(Tx_{2n+1},Tx_{2n}) \leq d(Tx_{2n},Tx_{2n-1})
\end{equation}
by substituting $x=x_{2n+1}$ and $y=x_{2n}$ in (2.2.1);

and similarly, by taking $x=x_{2n-1}$ and $y=x_{2n}$ we obtain
\begin{equation} \label{(2.7.2)}
  d(Tx_{2n-1},Tx_{2n}) \leq d(Tx_{2n-2},Tx_{2n-1})
\end{equation}
and concluded that, for all $n \geq 0$,
\begin{equation} \label{(2.7.3)}
  d(Tx_{n+1},Tx_n) \leq d(Tx_n,Tx_{n-1}).
\end{equation}

But if we take $x=x_{2n}$ and $y=x_{2n-1}$, for the substitution, then it is not possible to apply the inequality \eqref{(2.2.1)} to get \eqref{(2.7.3)}. The similar difficulty arises if we consider the inequality \eqref{(2.5.1)} in proving Theorem \eqref{thrm 2.6}. To avoid such difficulty, we consider the following inequalities in \eqref{(2.7.4)} and \eqref{(2.7.5)} in place of \eqref{(2.2.1)} and \eqref{(2.5.1)} respectively.
\end{remark}

\begin{equation} \label{(2.7.4)}
  d(Tx,Ty) \leq \varphi(\min\{d(fx,gy),d(fy,gx)\}) \quad \text{for all } x,y \in K;
\end{equation}
and
\begin{multline} \label{(2.7.5)}
  d(Tx,Ty) \leq r \min\Big\{\max\big\{d(fx,gy),d(Tx,fx),d(Ty,gy),\\
\tfrac{1}{2}[d(Tx,gy)d(Ty,+fx)]\big\}, \max\big\{d(gx,fy),d(Tx,gx),\\
d(Ty,gy),\tfrac{1}{2}[d(Tx,fy)+d(Ty,gx)]\big\}\Big\}
\end{multline}
for all $x,y \in K$.

The following lemmas are useful in our subsequent discussion.

\begin{lemma}[Sastry \emph{et al.} \cite{Sastry2003}, Remark 1, p. 18] \label{lem 2.8}
If $\alpha:(0,\infty)\to[0,1)$ is upper semicontinuous then there exists a continuous function $\beta:(0,\infty)\to[0,1)$ such that $\alpha(t)\leq\beta(t)$ for all $t$ in $(0,\infty)$.
\end{lemma}

\begin{lemma}[Sastry \emph{et al.} \cite{Sastry2003}] \label{lem 2.9}
Let $\varphi:[0,\infty)\to[0,\infty)$ be continuous such that $\varphi(t)<t$ for all $t\in(0,\infty)$. Then there exists a monotonically increasing continuous function $\psi:[0,\infty)\to[0,\infty)$ such that $\varphi(t)\leq\psi(t)$ for all $t\in \mathbb{R}_{+}$ and $\psi(t)<t$ for all $t>0$.
\end{lemma}

\begin{lemma}[Krishna Murthy \cite{KrishnaMurthy2006}] \label{lem 2.10}
Let $\varphi:\mathbb{R}_{+}\to \mathbb{R}_{+}$ be upper semicontinuous function such that $\varphi(t)<t$ for all $t>0$. Then there exists a monotonically increasing continuous function $\psi:\mathbb{R}_{+}\to \mathbb{R}_{+}$ such that $\varphi(t)\leq\psi(t)$ for all $t\in \mathbb{R}_{+}$ and $\psi(t)<t$ for all $t>0$.
\end{lemma}

\begin{proof}
Since $\varphi$ is upper semicontinuous on $(0,\infty)$ and $\varphi(t)<t$ for all $t\in(0,\infty)$, we have $\frac{\varphi(t)}{t}$ is upper semicontinuous from $(0,\infty)$ to $[0,1)$. Hence, by Lemma \eqref{lem 2.8}, there exists a continuous function $\beta:(0,\infty)\to[0,1)$ such that $\varphi(t)\leq t\beta(t)$ on $[0,1)$. Since $t\beta(t)$ is continuous on $[0,\infty)$, by Lemma \eqref{lem 2.9}, there exists a monotonically increasing continuous function $\psi:[0,\infty)\to[0,\infty)$ such that $\psi(t)<t$ and $t\beta(t)\leq\psi(t)$. Therefore, $\varphi(t)\leq\psi(t)$ and $\psi(t)<t$ on $[0,\infty)$.
\end{proof}

\section{Main Results}

The main result of this paper is the following.
\begin{theorem} \label{Thrm 3.1}
Let $K$ be a subset of a metric space $(X,d)$ and $T$, $f$, $g: K \to K$ be self-maps satisfying $\overline{T(K)} \subseteq f(K) \cap g(K)$. Assume:
\begin{multline} \label{(3.1.1)}
d(Tx,Ty) \leq \psi\Bigg(\min\Bigg\{\max\Big\{d(fx,gy),d(Tx,fx),d(Ty,gy),\\
\tfrac{1}{2}\left[d(Tx,gy) + d(Ty,fx)\right]\Big\}, \\
\max\Big\{d(fy,gx),d(Tx,gx),d(Ty,fy),\\
\tfrac{1}{2}\left[d(Tx,fy) + d(Ty,gx)\right]\Big\}\Bigg\}\Bigg)
\end{multline}
for all $x,y \in K$, where $\psi: \mathbb{R}_{+} \to \mathbb{R}_{+}$ is a monotonically increasing continuous function with $0 \leq \psi(t) < t$ for all $t > 0$. If either:
\begin{enumerate}
    \item[(i)] $\overline{T(K)}$ is complete, or
    \item[(ii)] $f(K)$ is complete, or
    \item[(iii)] $g(K)$ is complete
\end{enumerate}
and the pairs $(T,f)$, $(T,g)$ are weakly compatible, then there exists a unique $z \in K$ such that
\[
\{z\} = F(T) \cap F(f) \cap F(g).
\]
\end{theorem}

\begin{proof}
Let $x_0 \in K$. Since $\overline{T(K)} \subset f(K) \cap g(K)$, we can find a sequence $\{x_n\}$ in $K$ such that
\begin{equation}
y_{2n} = Tx_{2n} = fx_{2n+1} \quad \text{and} \quad y_{2n+1} = Tx_{2n+1} = gx_{2n+2}, \quad n=0,1,2,\ldots
\end{equation}

We write
\begin{equation}
\alpha_n = d(y_n, y_{n+1}), \quad n=0,1,2,\ldots
\end{equation}

In the sequence $\{y_n\}$, if $y_{2n} = y_{2n+1}$ for some $n$, then $fx_{2n+1} = Tx_{2n} = Tx_{2n+1} = gx_{2n+2}$. Using this identity and \eqref{(3.1.1)}, it follows that $y_{2n+2} = y_{2n+1}$. Thus $y_{2n+2} = y_{2n+1} = y_{2n}$, and in general, we have $y_m = y_n$ for all $m \geq n$, so that $\{y_m\}_{m \geq n}$ is a constant sequence and hence Cauchy. Therefore, without loss of generality, we assume that $y_n \neq y_{n+1}$ for all $n \in \mathbb{N}$.

Applying \eqref{(3.1.1)} to $x = x_{2n+1}, y = x_{2n}$:
\begin{align*}
\alpha_{2n} &= d(T x_{2n+1}, T x_{2n}) \\
&\leq \psi\Big(\min\Big\{\max\big\{d(fx_{2n+1}, gx_{2n}), d(Tx_{2n+1}, fx_{2n+1}), \\
&\quad d(Tx_{2n}, gx_{2n}), \tfrac{1}{2}[d(Tx_{2n+1}, gx_{2n}) + d(Tx_{2n}, fx_{2n+1})]\big\}, \\
&\quad \max\big\{d(gx_{2n+1}, fx_{2n}), d(Tx_{2n}, fx_{2n}), d(Tx_{2n+1}, gx_{2n+1}), \\
&\quad \tfrac{1}{2}[d(Tx_{2n}, gx_{2n+1}) + d(Tx_{2n+1}, fx_{2n})]\big\}\Big\}\Big) \\
&\leq \psi \Big( \max \big\{ d(f x_{2n+1}, g x_{2n}), d(T x_{2n+1}, f x_{2n+1}), \\
&\quad d(T x_{2n}, g x_{2n}), \tfrac{1}{2}[d(T x_{2n+1}, g x_{2n}) + d(T x_{2n}, f x_{2n+1})] \big\} \Big)
\end{align*}
Substituting $f x_{2n+1} = y_{2n}$, $g x_{2n} = y_{2n-1}$ (for $n \geq 1$), $T x_{2n+1} = y_{2n+1}$, $T x_{2n} = y_{2n}$:
\begin{align*}
d(f x_{2n+1}, g x_{2n}) &= d(y_{2n}, y_{2n-1}) = \alpha_{2n-1} \\
d(T x_{2n+1}, f x_{2n+1}) &= d(y_{2n+1}, y_{2n}) = \alpha_{2n} \\
d(T x_{2n}, g x_{2n}) &= d(y_{2n}, y_{2n-1}) = \alpha_{2n-1} \\
\tfrac{1}{2}[d(T x_{2n+1}, g x_{2n}) + d(T x_{2n}, f x_{2n+1})] &= \tfrac{1}{2}[d(y_{2n+1}, y_{2n-1}) + d(y_{2n}, y_{2n})] \\
&= \tfrac{1}{2} d(y_{2n+1}, y_{2n-1}) \leq \tfrac{1}{2} (\alpha_{2n} + \alpha_{2n-1})
\end{align*}
Thus:
\[
\alpha_{2n} \leq \psi \left( \max \left\{ \alpha_{2n-1}, \alpha_{2n}, \tfrac{1}{2}(\alpha_{2n} + \alpha_{2n-1}) \right\} \right) \leq \psi \left( \max \{\alpha_{2n-1}, \alpha_{2n}\} \right)
\]
If $\max\{\alpha_{2n-1}, \alpha_{2n}\} = \alpha_{2n}$, then $\alpha_{2n} \leq \psi(\alpha_{2n}) < \alpha_{2n}$, contradiction. Hence:
\begin{equation} \label{eq:even-dec}
\alpha_{2n} \leq \psi(\alpha_{2n-1}) < \alpha_{2n-1}
\end{equation}
Similarly, applying \eqref{(3.1.1)} to $x = x_{2n+2}, y = x_{2n+1}$:
\begin{equation} \label{eq:odd-dec}
\alpha_{2n+1} \leq \psi(\alpha_{2n}) < \alpha_{2n}
\end{equation}
Thus $\{\alpha_n\}$ is decreasing and bounded below, so converges to some $r \geq 0$. If $r > 0$, continuity of $\psi$ gives $r \leq \psi(r) < r$, contradiction. Hence:
\begin{equation} \label{eq:dist-conv}
\lim_{n \to \infty} d(y_n, y_{n+1}) = 0
\end{equation}

Next we show that \textbf{$\{y_n\}$ is Cauchy}.
Suppose $\{y_n\}$ is not Cauchy. Then $\exists \varepsilon > 0$ and subsequences $\{m_k\}, \{n_k\}$ with $m_k > n_k > k$ such that:
\begin{align} 
d(y_{2m_k}, y_{2n_k}) &\geq \varepsilon \label{eq:cauchy1} \\
d(y_{2m_k-2}, y_{2n_k}) &< \varepsilon \label{eq:cauchy2}
\end{align}
By triangle inequality:
\[
\varepsilon \leq d(y_{2m_k}, y_{2n_k}) \leq d(y_{2m_k}, y_{2m_k-1}) + d(y_{2m_k-1}, y_{2m_k-2}) + d(y_{2m_k-2}, y_{2n_k})
\]
Using \eqref{eq:dist-conv} and \eqref{eq:cauchy2}:
\begin{equation} \label{eq:lim1}
\lim_{k \to \infty} d(y_{2m_k}, y_{2n_k}) = \varepsilon
\end{equation}
Similarly:
\begin{align}
\lim_{k \to \infty} d(y_{2m_k}, y_{2n_k+1}) &= \varepsilon \label{eq:lim2} \\
\lim_{k \to \infty} d(y_{2m_k-1}, y_{2n_k}) &= \varepsilon \label{eq:lim3} \\
\lim_{k \to \infty} d(y_{2m_k-1}, y_{2n_k+1}) &= \varepsilon \label{eq:lim4}
\end{align}

Now apply \eqref{(3.1.1)} to $x = x_{2n_k+1}, y = x_{2m_k}$:
\begin{align*}
d(y_{2n_k+1}, y_{2m_k}) &\leq \psi \left( \max \left\{ d(y_{2n_k}, y_{2m_k-1}), \alpha_{2n_k}, \alpha_{2m_k-1}, \tfrac{1}{2} d(y_{2n_k+1}, y_{2m_k-1}) \right\} \right)
\end{align*}
Taking $k \to \infty$ and using \eqref{eq:lim1}-\eqref{eq:lim4}:
\[
\varepsilon \leq \psi \left( \max \left\{ \varepsilon, 0, 0, \tfrac{\varepsilon}{2} \right\} \right) = \psi(\varepsilon) < \varepsilon
\]
Contradiction. Hence $\{y_n\}$ is Cauchy in $\overline{T(K)}$. Consequently, $\{Tx_n\}$ is Cauchy, and hence the sequences $\{fx_{2n+1}\}$ and $\{gx_{2n+2}\}$ are Cauchy sequences in $\overline{T(K)} \subset f(K) \cap g(K)$. Since $\{y_n\} \subseteq T(K) \subseteq \overline{T(K)} \subseteq f(K) \cap g(K)$ and one of $\overline{T(K)}$, $f(K)$, or $g(K)$ is complete, $\exists z \in X$ with $\lim y_n = z$.

\subsection*{Case (i): $\overline{T(K)}$ is complete}
Then $Tx_n \to z$ (say) as $n\to\infty$, $z\in\overline{T(K)}$. By the definition of $\{Tx_n\}$, we have
\[
gx_{2n+2}, fx_{2n+1} \to z \quad \text{as} \quad n\to\infty.
\]

Since $\overline{T(K)} \subset f(K) \cap g(K)$, there exist $u,v \in K$ such that $fu = z = gv$.

From inequality \eqref{(3.1.1)}:
\begin{align*}
d(Tx_{2n+1},Tv) &\leq \psi\Big(\min\Big\{\max\big\{d(fx_{2n+1},gv), d(Tx_{2n+1},fx_{2n+1}), d(Tv,gv), \\
&\quad \tfrac{1}{2}[d(Tx_{2n+1},gv) + d(Tv,fx_{2n+1})]\big\}, \max\big\{d(gx_{2n+1},fv), \\
&\quad d(Tv,fv), d(Tx_{2n+1},gx_{2n+1}), \\
&\quad \tfrac{1}{2}[d(Tv,gx_{2n+1}) + d(Tx_{2n+1},fv)]\big\}\Big\}\Big)\\
&\leq \psi\Big(\max\big\{d(fx_{2n+1},gv), d(Tx_{2n+1},fx_{2n+1}), d(Tv,gv), \\
&\quad \tfrac{1}{2}[d(Tx_{2n+1},gv) + d(Tv,fx_{2n+1})]\big\}\Big).
\end{align*}

Letting $n\to\infty$ and using the continuity of $\psi$:
\begin{align*}
d(z,Tv) &\leq \psi\Big(\max\big\{d(z,gv), d(z,z), d(Tv,gv), \tfrac{1}{2}[d(z,gv) + d(Tv,z)]\big\}\Big) \\
&= \psi\Big(\max\big\{0, d(Tv,z), \tfrac{1}{2}[d(z,z) + d(Tv,z)]\big\}\Big) \\
&= \psi\Big(\max\big\{d(z,z), d(Tv,z), \tfrac{1}{2}d(Tv,z)\big\}\Big) \\
&= \psi(d(Tv,z)) < d(Tv,z),
\end{align*}
a contradiction. Hence $z=Tv$, i.e., $Tv=z=gv$.

Similarly, $Tu=z=fu$.

Since $(T,f)$ and $(T,g)$ are weakly compatible and $Tu=fu=z=gv=Tv$, we have $Tfu=fTu$ and $Tgv=gTv$. Hence
\begin{equation}\label{(3.1.13)}
fz = Tz = gz.
\end{equation}

We claim that $z$ is a common fixed point of $T$, $f$, and $g$.

If $z \neq Tz$, then from \eqref{(3.1.1)}:
\begin{align*}
d(z,Tz) &= d(Tu,Tz) \\
&\leq \psi\Big(\min\Big\{\max\big\{d(fu,gz), d(Tu,fu), d(Tz,gz), \\
&\quad \tfrac{1}{2}[d(Tu,gz) + d(Tz,fu)]\big\}, \max\big\{d(gu,fz), d(Tz,fz), \\
&\quad d(Tu,gu), \tfrac{1}{2}[d(Tz,gu) + d(Tu,fz)]\big\}\Big\}\Big) \\
&\leq \psi\Big(\max\big\{d(fu,gz), d(Tu,fu), d(Tv,gz), \\
&\quad \tfrac{1}{2}[d(Tu,gz) + d(Tz,fu)]\big\}\Big) \\
&= \psi\Big(\max\big\{d(z,Tz), d(z,z), d(z,Tz), \\
&\quad \tfrac{1}{2}[d(z,Tz) + d(Tz,z)]\big\}\Big) \\
&= \psi(d(Tz,z)).
\end{align*}
Thus $Tz=z$.
From \eqref{(3.1.13)}:
\[
z \in F(T) \cap F(f) \cap F(g).
\]
The uniqueness follows from \eqref{(3.1.1)}.

\subsection*{Case (ii): $f(K)$ is complete}
Since $\{y_n\}$ is Cauchy, its subsequence $\{fx_{2n+1}\}$ is a Cauchy sequence in $f(K)$. As $f(K)$ is complete, there exists $z \in f(K)$ such that $fx_{2n+1} \to z$. Since the entire sequence $\{y_n\}$ is Cauchy and contains a convergent subsequence, $y_n \to z$. Moreover, since $\{y_n\} \subseteq T(K)$, we have $z \in \overline{T(K)} \subset f(K) \cap g(K)$. Thus there exist $u,v \in K$ such that $fu = z = gv$. The conclusion now follows exactly as in Case (i).

\subsection*{Case (iii): $g(K)$ is complete}
The proof is symmetric to Case (ii) using the subsequence $\{gx_{2n+2}\}$.

Thus in all cases, $F(T) \cap F(f) \cap F(g) = \{z\}$.
\end{proof}

Now we relax the monotonicity and continuity of $\psi$ by upper semi-continuity as follows.

\begin{corollary}
Let $K$ be a subset of a metric space $(X,d)$, and $T, f, g: K \to K$ be three mappings which satisfy $\overline{T(K)} \subset f(K) \bigcap g(K)$. Assume that
\begin{equation}
d(Tx, Ty) \leq \varphi \left( \min \left\{ 
\begin{aligned}
&\max \left\{ d(fx, gy), d(Tx, fx), d(Ty, gy), \frac{1}{2}[d(Tx, gy) + d(Ty, fx)] \right\}, \\
&\max \left\{ d(fy, gx), d(Tx, gx), d(Ty, fy), \frac{1}{2}[d(Tx, fy) + d(Ty, gx)] \right\}
\end{aligned}
\right\} \right)
\end{equation}
for all $x, y \in K$; where $\varphi: \mathbb{R}_{+} \to \mathbb{R}_{+}$ is an upper semi-continuous function on $\mathbb{R}_{+}$ and $0 \leq \varphi(t) < t$ for $t > 0$. If either $\overline{T(K)}$ or $f(K)$ or $g(K)$ is complete and the pair of mappings $(T,f)$ and $(T,g)$ are weakly compatible, then $F(T) \bigcap F(f) \bigcap F(g)$ is a singleton.
\end{corollary}
\begin{proof}
    The proof follows from Lemma \eqref{lem 2.9} and Theorem \eqref{Thrm 3.1}.
\end{proof}

\begin{example}
    
Let $X = (0,1)$ with the usual metric and let $K = \left( \frac{1}{3}, 1 \right)$. We define self-maps $T$, $f$, and $g$ on $K$ by

 \[
    T(x) = \begin{cases} 
        \frac{1}{2} & \frac{1}{3} < x < \frac{2}{3} \\
        \frac{2}{3} & \frac{2}{3} \leq x < 1 
    \end{cases}, \quad
    f(x) = \begin{cases} 
        \frac{5}{6} & \frac{1}{3} < x < \frac{2}{3} \\
        \frac{4}{3} - x & \frac{2}{3} \leq x < 1 
    \end{cases}, \quad \text{and} \quad
    g(x) = \begin{cases} 
        \frac{5}{6} & \frac{1}{3} < x < \frac{2}{3} \\
        \frac{2}{3} & x = \frac{2}{3} \\
        \frac{1}{2} & \frac{2}{3} < x < 1 
    \end{cases}
    \]

Let \(\varphi: \mathbb{R}_{+} \to \mathbb{R}_{+}\) be defined by
\[
\varphi(t) = 
\begin{cases}
\frac{1}{2}t & \text{if } 0 < t < 1, \\
\frac{2}{3} & \text{if } t = 1, \\
-\frac{1}{2}t + 1 & \text{if } 1 < t < 2, \\
0 & \text{if } t > 2.
\end{cases}
\]
Then \(\varphi(t)\) is upper semi-continuous and \(\varphi(t) < t\) for all \(t > 0\). Now we consider \(\psi: \mathbb{R}_{+} \to \mathbb{R}_{+}\) defined by \(\psi(t) = \frac{2}{3}t\). Then \(\psi\) is continuous, monotone increasing, and satisfies \(\varphi(t) \leq \psi(t)\) for all \(t \in [0, \infty)\) and \(\psi(t) < t\) for all \(t > 0\).

Here \(\overline{T(K)} = \left\{\frac{1}{2}, \frac{2}{3}\right\}\), \(f(K) = \left(\frac{1}{3}, \frac{2}{3}\right] \bigcup \left\{\frac{5}{6}\right\}\), and \(g(K) = \left\{\frac{1}{2}, \frac{2}{3}, \frac{5}{6}\right\}\), so that \(\overline{T(K)} \subset f(K) \bigcap g(K)\). The pairs of mappings \((T,f)\) and \((T,g)\) are weakly compatible on \(K\). Also, the self-maps \(T\), \(f\), and \(g\) and the map \(\psi\) satisfy the inequality \eqref{(3.1.1)}. Hence, \(T\), \(f\), and \(g\) and \(\psi\) satisfy all the hypotheses of Theorem \eqref{Thrm 3.1}, and \(\frac{2}{3}\) is the unique common fixed point of \(T\), \(f\), and \(g\).
\end{example}
\begin{corollary} \label{Cor 3.4}
Let \(K\) be a subset of a metric space \((X,d)\) and \(T, f: K \to K\) be two mappings such that \(\overline{T(K)} \subset f(K)\). Assume that
\begin{equation} \label{(3.4.1)}
d(Tx, Ty) \leq \varphi \left( \max \left\{ d(fx, fy), d(Tx, fx), d(Ty, fy), \frac{1}{2}[d(Tx, fy) + d(Ty, fx)] \right\} \right)
\end{equation}
for all \(x, y \in K\); where \(\varphi: \mathbb{R}_{+} \to \mathbb{R}_{+}\) is a monotonically increasing and continuous function on \(\mathbb{R}_{+}\) and \(0 \leq \varphi(t) < t\) for \(t > 0\). If either \(\overline{T(K)}\) or \(f(K)\) is complete and the pair of mappings \((T,f)\) is weakly compatible, then \(F(T) \bigcap F(f)\) is a singleton.
\end{corollary}
\begin{proof}
    Follows from Theorem \eqref{Thrm 3.1} by taking \(g = f\) on \(K\).
\end{proof}

\begin{remark}
    In Corollary \eqref{Cor 3.4}, it suffices to take \(T(K) \subset f(K)\) in place of \(\overline{T(K)} \subset f(K)\) since if \(f(K)\) is complete, then it is closed and hence \(\overline{T(K)} \subset f(K)\) follows.
\end{remark}

The following example supports Corollary \eqref{Cor 3.4}.

\begin{example} \label{Ex 3.4}
Let \(X = (0,1]\) with the usual metric and \(K = \left(\frac{1}{3}, 1\right)\). We define self-maps \(T\) and \(f\) by
\[
    T(x) = \begin{cases} 
        \frac{1}{2} & \frac{1}{3} < x < \frac{2}{3} \\
        1 - \frac{1}{2}x & \frac{2}{3} \leq x < 1 
    \end{cases}, \quad \text{and} \quad
    f(x) = \begin{cases} 
        \frac{5}{6} & \frac{1}{3} < x < \frac{2}{3} \\
        \frac{4}{3} - x & \frac{2}{3} \leq x < 1 
    \end{cases}
    \]

We have \(T(K)=\left[\frac{1}{2},\frac{2}{3}\right]\) and \(f(K)=\left(\frac{1}{3},\frac{2}{3}\right]\bigcup\left\{\frac{5}{6}\right\}\) so that \(T(K)\subset f(K)\), and \(f\) and \(T\) are weakly compatible on \(K\). Also the selfmaps \(f\) and \(T\) and the function \(\varphi:\mathbb{R}_{+}\to \mathbb{R}_{+}\) defined by \(\varphi(t)=\frac{1}{2}t\) satisfy the inequality \eqref{(3.4.1)}. Hence \(f\) and \(T\) satisfy all the hypotheses of Corollary \eqref{Cor 3.4} and \(\frac{2}{3}\) is the unique common fixed point of \(T\) and \(f\).

We observe that the pair of selfmaps \((T,f)\) is not compatible on \(K\); for, take \(x_{n}=\frac{2}{3}+\frac{1}{n}\), \(n=5,6,7,\ldots\). Then \(Tx_{n}=\frac{2}{3}-\frac{1}{2n}\) and \(fx_{n}=\frac{2}{3}-\frac{1}{n}\), \(n=5,6,7,\ldots\). We note that \(Tx_{n}\rightarrow\frac{2}{3}\) as \(n\rightarrow\infty\) and \(fx_{n}\rightarrow\frac{2}{3}\) as \(n\rightarrow\infty\). Now \(T(fx_{n})=T\left(\frac{2}{3}-\frac{1}{n}\right)=\frac{1}{2}\) and \(f(Tx_{n})=f\left(\frac{2}{3}-\frac{1}{2n}\right)=\frac{5}{6}\) so that 
\[
\lim\limits_{n\rightarrow\infty}Tfx_{n}=\frac{1}{2}\neq\frac{5}{6}=\lim\limits_{n\rightarrow\infty}fTx_{n}.
\]

Also we observe that \(T\) and \(f\) are not reciprocal continuous on \(K\), for 
\[
\lim\limits_{n\rightarrow\infty}Tfx_{n}=\frac{1}{2}\neq\frac{2}{3}=T\left(\frac{2}{3}\right)
\]
and 
\[
\lim\limits_{n\rightarrow\infty}fTx_{n}=\frac{5}{6}\neq\frac{2}{3}=f\left(\frac{2}{3}\right).
\]

Hence this example shows that Corollary \eqref{Cor 3.4} which in turn Theorem \eqref{Thrm 3.1} is a generalization of Theorem \eqref{thrm 1.1}, Theorem \eqref{thrm 1.2}, Theorem \eqref{thrm 1.5} and Theorem \eqref{thrm 1.8}.

Interestingly the functions \(T\) and \(f\) do not satisfy the contraction condition of Al-Thagafi \textit{et. al.} [1], Theorem \eqref{thrm 2.3}, for when \(x=\frac{2}{3}\) and \(y\in\left(\frac{1}{3},\frac{2}{3}\right)\), \(|Tx-Ty|=\frac{1}{6}\) and \(|fx-fy|=\frac{1}{6}\). So, \(d(Tx,Ty)\leq r d(fx,fy)\) does not hold for any \(0<r<1\). Hence corollary \eqref{Cor 3.4} is a generalization of Al-Thagafi and Shahzad Theorem (Theorem \eqref{thrm 2.3}) too.
\end{example}
\begin{remark}
    Theorem \eqref{thrm 1.7} follows as a corollary to Corollary \eqref{Cor 3.4}, since the inequality \eqref{(1.7.1)} implies the inequality \eqref{(3.4.1)}. Example \eqref{Ex 3.4} shows that the mappings \(T\) and \(f\) do not satisfy the contraction conditions of Theorem \eqref{thrm 1.7}, for when \(x=\frac{2}{3}\) and \(y\in\left(\frac{1}{3},\frac{2}{3}\right)\), we have \(|Tx-Ty|=\frac{1}{6}\), \(|fx-fy|=\frac{1}{6}\), \(|Tx-fy|=\frac{1}{6}\) and \(|fx-Ty|=\frac{1}{6}\) so that the inequality \eqref{(1.7.1)} does not hold for any \(r>0\).
\end{remark}

\begin{remark}
   The requirement of the completeness of \(f(K)\) or \(T(K)\) is essential in the results so obtained and can not be relaxed.
\end{remark}

\begin{example}
    Let \(X=\left\{0,1,\frac{1}{2},\frac{1}{2^{2}},\frac{1}{2^{3}},...\right\}\) be a metric with the usual metric. We define selfmaps \(T\), \(f:X\to X\) by
    \[
    Tx = \begin{cases} 
        \frac{1}{4} & x = 0 \\
        \frac{1}{4}x & x \neq 0 
    \end{cases} \quad \text{and}, \quad
    fx = \begin{cases} 
        \frac{1}{2} & x = 0 \\
        \frac{1}{2}x & x \neq 0 
    \end{cases}
    \]

We can verify that all the conditions of Corollary \eqref{Cor 3.4} are satisfied except the completeness of the space \(f(K)\) and \(T(K)\). Note that \(T\) and \(f\) have no common fixed points.
\end{example}

\section{Additional Results}

\begin{theorem}\label{thrm 4.1}
    Let \(K\) be a subset of a metric space \((X,d)\). Let \(T\), \(f\) be two selfmappings of \(K\) with
\begin{equation} \label{(4.1.1)}
d(T^{m}x,T^{m}y) \leq \varphi\left(\max\left\{
\begin{aligned}
&d(fx,fy), d(T^{m}x,fx), d(T^{m}y, fy), \\
&\frac{1}{2}[d(T^{m}x,fy) + d(fx,T^{m}y)]
\end{aligned}
\right\}\right)
\end{equation}
for all \(x,y \in K\), where \(\varphi:\mathbb{R}_{+} \to \mathbb{R}_{+}\) is monotonically increasing and continuous on \(\mathbb{R}_{+}\) and \(0 \leq \varphi(t) < t\) for \(t > 0\). Further assume that either \(T^{m}(K)\) or \(f(K)\) is complete and the pair of mappings \((T^{m},f)\) is weakly compatible for some \(m \in \mathbb{N}\). Then \(T^{m}\) and \(f\) have a unique common fixed point. Further, if \(F(T^{m}) \cap F(f) \neq \emptyset\), and if \(T\) and \(f\) commute on \(F(T^{m}) \cap F(f)\), then \(f\) and \(T\) have a unique common fixed point in \(K\).
\end{theorem}

\begin{proof}
    Since \(T(K) \subset K\), we have \(T^{n}(K) \subset K\) for all \(n \in \mathbb{N}\) and hence \(T^{m}\) is a selfmap on \(K\). Since \(T^{m}\) and \(f\), for some \(m \in \mathbb{N}\), satisfy all the hypotheses of Corollary \eqref{Cor 3.4}, \(F(T^{m}) \cap F(f)\) is a singleton. Hence there is a unique \(z \in K\) such that
\begin{equation} \label{(4.1.2)}
T^{m}(z) = f(z) = z
\end{equation}

Now by applying \(T\) to both sides of \(T^{m}(z) = z\), we get
\begin{equation} \label{(4.1.3)}
T^{m}(T(z)) = T(z)
\end{equation}
Hence, \(T(z)\) is also a fixed point of \(T^{m}\).

Now we claim that \(T(z) = z\).

Since \(F(T) \subset F(T^{n})\) for any \(n \in \mathbb{N}\), we have \(F(T) \subset F(T^{m})\). This implies that \(F(T) \cap F(f) \subset F(T^{m}) \cap F(f)\). But \(F(T^{m}) \cap F(f) = \{z\}\).

Hence,
\begin{equation} \label{(4.1.4)}
F(T) \cap F(f) \subset \{z\}
\end{equation}

We now show that \(F(T) \cap F(f) \neq \emptyset\) so that from the inclusion \eqref{(4.1.4)}, we get \(F(T) \cap F(f) = \{z\}\).

If \(Tz \neq z\), then using inequality \eqref{(4.1.1)} and using \eqref{(4.1.2)} and \eqref{(4.1.3)}, we have
\begin{align*}
d(Tz, z) &= d(T^{m}(Tz), T^{m}z) \\
&\leq \varphi\left(\max\left\{
\begin{aligned}
&d(fTz, fz), d(T^{m}(Tz), f(Tz)), d(T^{m}z, fz), \\
&\frac{1}{2}[d(T^{m}(Tz), fz) + d(f(Tz), T^{m}z)]
\end{aligned}
\right\}\right) \\
&= \varphi\left(\max\left\{
\begin{aligned}
&d(fTz, z), d(Tz, f(Tz)), d(z, z), \\
&\frac{1}{2}[d(Tz, z) + d(f(Tz), z)]
\end{aligned}
\right\}\right) \\
&\leq \varphi\left(\max\left\{
d(fTz, z), d(Tz, f(Tz)), 0, \frac{1}{2}[2 \max \{d(Tz, z), d(f(Tz), z)\}]
\right\}\right) \\
&= \varphi\left(\max\left\{d(fTz, z), d(Tz, fTz),  d(Tz, z)\right\}\right).
\end{align*}

This implies that
\begin{equation}
d(Tz, z) \leq \varphi\left(\max\left\{d(fTz, z), d(Tz, fTz), 0, d(Tz, z)\right\}\right).
\end{equation}

Now since \(T\) and \(f\) commute on \(F(T^{m}) \cap F(f)\), from \eqref{(4.5.1)} we get
\begin{align*}
d(Tz, z) &\leq \varphi\left(\max\left\{d(Tfz, z), d(Tz, Tfz),  d(Tz, z)\right\}\right) \\
&= \varphi(d(Tz, z)) < d(Tz, z),
\end{align*}
a contradiction. Hence \(Tz = z\) so that \(F(T) \cap F(f) \neq \emptyset\).

Hence \(z\) is the unique common fixed point of \(T\) and \(f\).
\end{proof}

\begin{example}
    Let \(X = \left(\frac{1}{4}, 4\right)\) with the usual metric and let \(K = \left(\frac{1}{4}, 3\right]\). We define mappings \(T\) and \(f\) on \(K\) by

     \[
    T(x) = \begin{cases} 
        1 & \frac{1}{4} < x < \frac{1}{3} \\
        \frac{1}{x} & \frac{1}{3} \leq x \leq \frac{23}{15} \\
        1 & \frac{23}{15} < x \leq 3 
    \end{cases}, \quad \text{and} \quad
    f(x) = \begin{cases} 
        1 & \frac{1}{4} < x < \frac{1}{3} \\
        -\frac{5}{4}x + \frac{9}{4} & \frac{1}{3} \leq x \leq \frac{23}{15} \\
        1 & \frac{23}{15} < x \leq 3 
    \end{cases}
    \]

Here we observe that \(T\) and \(f\) do not satisfy the inequality \eqref{(3.4.1)}, for, taking \(x = 1\), and \(y = \frac{1}{3}\), then \(|Tx - Ty| = 2\), \(|fx - fy| = \frac{5}{6}\), \(|x - Ty| = 0\), \(|fy - Ty| = \)

We have $\frac{7}{6}$, and $\frac{1}{2}\left(|Tx-fy| + |fx-Ty|\right) = \frac{17}{12}$. So, the inequality \eqref{(3.4.1)} does not hold for any $r$ in $[0,1)$.

Since $T^2:K\to K$ defined by
\[
T^2x = \left\{
\begin{array}{ccc}
1 & \text{if} & \frac{1}{4}<x<\frac{1}{3} \\ 
\\ 
x & \text{if} & \frac{1}{3}\leq x\leq\frac{23}{15} \\ 
\\ 
1 & \text{if} & \frac{23}{15}<x\leq 3;
\end{array}
\right.
\]
with $T^2(K)=\left\{\frac{1}{3},\frac{23}{15}\right\}$ and $f(K)=\left\{\frac{1}{3},\frac{22}{12}\right\}$ so that $T^2(K)\subset f(K)$, and $f$ and $T^2$ are weakly compatible on $K$. Also the maps $f$ and $T^2$ satisfy the inequality \eqref{(4.1.1)} with $\varphi:\mathbb{R}_{+}\to \mathbb{R}_{+}$ defined by $\varphi(t)=\frac{4}{5}t$ and hence $f$ and $T^2$ satisfy all the hypotheses of Theorem \eqref{thrm 4.1}, and 1 is the unique common fixed point of $T^2$ and $f$. Moreover, $fT=Tf$ on $F(T^2)\cap F(f)$. Hence, 1 is also the unique common fixed point of $T$ and $f$.
\end{example}

\begin{remark}
    In Theorem \eqref{thrm 4.1}, if $T$ and $f$ do not commute on $F(T^m)\cap F(f)$ for any $m\in \mathbb{N}$, then $T$ and $f$ may not have a common fixed point, as illustrated by the following example.
\end{remark}

\begin{example}
    Let $X=(0,2]$ with the usual metric and let $K=\left(\frac{1}{3},2\right]$. We define $T,f:K\to K$ by

    \[
    Tx = \begin{cases} 
        2 & \frac{1}{3} < x < 1 \\
        \frac{1}{x} & 1 \leq x \leq 2 
    \end{cases} \quad \text{and},
    fx = \begin{cases} 
        2 & \frac{1}{3} < x < 1 \\
        \frac{3}{2}x - 1 & 1 \leq x \leq 2 
    \end{cases}
    \]

Since $T^2:K\to K$ defined by
\[
T^2x = \left\{
\begin{array}{ccc}
\frac{1}{2} & \text{if} & \frac{1}{3}<x<1 \\ 
\\ 
x & \text{if} & 1\leq x\leq 2;
\end{array}
\right.
\]
with $T^2(K)=\left\{\frac{1}{2}\right\}\cup[1,2]$ and $f(K)=\left[\frac{1}{2},2\right]$ so that $T^2(K)\subset f(K)$, and $f$ and $T^2$ are weakly compatible on $K$. Also the maps $f$ and $T^2$ satisfy the inequality \eqref{(4.1.1)} with $\varphi:\mathbb{R}_{+}\to \mathbb{R}_{+}$ defined by $\varphi(t)=\frac{2}{3}t$ and hence $f$ and $T^2$ satisfy all the hypotheses of Theorem \eqref{thrm 4.1}.

All the hypotheses of Theorem \eqref{thrm 4.1} are satisfied, and $2$ is the unique common fixed point of $T^{2}$ and $f$. We have
\[
F(T^{m}) \cap F(f) = \left\{
\begin{array}{cc}
\{2\} & \text{if } m \text{ is even} \\ 
\\ 
\emptyset & \text{if } m \text{ is odd.}
\end{array}
\right.
\]

Hence $F(T^{m}) \cap F(f) \neq \emptyset$ for all $m$ even, and $Tf(2) = T(2) = \frac{1}{2}$ while $fT(2) = f\left(\frac{1}{2}\right) = 2$, so that $T$ and $f$ do not commute on $F(T^{m}) \cap F(f)$ for any $m$ even; consequently, $T$ and $f$ have no common fixed points.
\end{example}

We now extend Corollary \eqref{Cor 3.4} to a sequence of selfmaps.

\begin{theorem} \label{thrm 4.5}
    Let $\{T_{n}\}_{n=1}^{\infty}$ and $f$ be selfmaps of a subset $K$ of a metric space $(X, d)$ satisfying the condition
\begin{equation}\label{(4.5.1)}
T_{1}(K) \subset f(K).
\end{equation}
Assume that
\begin{equation} \label{(4.5.2)}
d(T_{1}x, T_{j}y) \leq \varphi\left(\max\left\{
\begin{aligned}
&d(fx, fy), d(T_{1}x, fx), d(T_{j}y, fy), \\
&\tfrac{1}{2}[d(T_{1}x, fy) + d(T_{j}y, fx)]
\end{aligned}
\right\}\right)
\end{equation}
for all $x, y$ in $K$ and for all $j$, where $\varphi: \mathbb{R}_{+} \to \mathbb{R}_{+}$ is monotonically increasing and continuous on $\mathbb{R}_{+}$ and $0 \leq \varphi(t) < t$ for $t > 0$. Further assume that either $T_{1}(K)$ or $f(K)$ is complete and the pair of mappings $(T_{1}, f)$ is weakly compatible. Then $\{T_{n}\}_{n=1}^{\infty}$ and $f$ have a unique common fixed point in $K$.
\end{theorem}

\begin{proof}
    By Corollary \eqref{Cor 3.4}, $T_{1}$ and $f$ have a unique fixed point $z$ in $K$. Thus,
\begin{equation} \label{(4.5.3)}
T_{1}z = fz = z.
\end{equation}

Now, let $j \in \mathbb{N}$ with $j \neq 1$. Then from the inequality \eqref{(4.5.2)} and using \eqref{(4.5.3)}, we get
\begin{align*}
d(T_{1}z, T_{j}z) &\leq \varphi\left(\max\left\{
\begin{aligned}
&d(fz, fz), d(T_{1}z, fz), d(T_{j}z, fz), \\
&\tfrac{1}{2}[d(T_{1}z, fz) + d(T_{j}z, fz)]
\end{aligned}
\right\}\right) \\
&= \varphi\left(\max\left\{d(z, z), d(z, z), d(T_{j}z, z), \tfrac{1}{2}[d(z, z) + d(T_{j}z, z)]\right\}\right) \\
&= \varphi(d(T_{j}z, z)) < d(T_{j}z, z),
\end{align*}
a contradiction. Hence $T_{j}z = z$.

The uniqueness of $z$ follows from inequality \eqref{(4.5.2)}. Therefore, $z$ is the unique common fixed point of the sequence of selfmaps $\{T_{n}\}_{n=1}^{\infty}$ and $f$.
\end{proof}

\begin{remark}
    Theorem \eqref{Thrm 1.6} follows as a corollary to Theorem \eqref{thrm 4.5} in the sense that inequality \eqref{(1.6.1)} implies inequality \eqref{(4.5.2)} with $\varphi:\mathbb{R}_{+}\to \mathbb{R}_{+}$ defined by $\varphi(t)=rt$, where $r=a_{1}+a_{2}+a_{3}+a_{4}+a_{5}<1$. Also, Theorem \eqref{thrm 4.5} shows that the condition $T_{n}(X)\subset T(X)$ for all $n$ in Theorem 3 of Som \cite{Som2003} is redundant. In fact, $T_{1}(K)\subset T(K)$ is sufficient.
\end{remark}

\begin{example}
    Let $X=(0,1]$ with the usual metric and $K=\left(\frac{1}{3},1\right]$. We define $\{T_{n}\}_{n=1}^{\infty}$ and $f$ on $K$ by
    \[
    T_n x = \begin{cases} 
        \frac{1}{3} + \frac{1}{6n} & \frac{1}{3} < x < \frac{2}{3} \\
        1 - \frac{1}{2}x & \frac{2}{3} \leq x < 1 
    \end{cases}, \quad
    f(x) = \begin{cases} 
        1 & \frac{1}{3} < x < \frac{2}{3} \\
        \frac{4}{3} - x & \frac{2}{3} \leq x < 1 
    \end{cases}
    \]

Here $T_{1}(K)=\left[\frac{1}{2},\frac{2}{3}\right]$ and $f(K)=\left(\frac{1}{3},\frac{2}{3}\right]\bigcup\{1\}$ so that $T(K)\subset f(K)$, and for each $n$, $f$ and $T_{n}$ are weakly compatible on $K$. Also for each $n$, the maps $f$ and $T_{n}$ satisfy the inequality \eqref{(4.5.2)} with $\varphi:\mathbb{R}_{+}\to \mathbb{R}_{+}$ defined by $\varphi(t)=\frac{1}{2}t$ and $f$ and $T_{n}$ satisfy all the hypotheses of Theorem \eqref{thrm 4.5}, and $\frac{2}{3}$ is the unique common fixed point of $T_{n}$ and $f$, $n=1,2,3,...$.

We observe that the pair of maps $(T_{n},f)$ is not compatible on $K$; for, take $x_{n}=\frac{2}{3}+\frac{1}{n}$, $n=4,5,6,...$. Then for each $n$, we have $fx_{n}=\frac{2}{3}-\frac{1}{n}$ and $T_{j}x_{n}=\frac{2}{3}-\frac{1}{2n}$. Note that $T_{j}x_{n}\rightarrow\frac{2}{3}$ as $n\rightarrow\infty$ and $fx_{n}\rightarrow\frac{2}{3}$ as $n\rightarrow\infty$. Now $T_{j}(fx_{n})=\frac{1}{3}+\frac{1}{6j}$ and $f(T_{j}x_{n})=1$ for each $j=1,2,3,...$ so that
\[
\lim_{n\rightarrow\infty}T_{j}(fx_{n})=\frac{1}{3}+\frac{1}{6j}\neq 1=\lim_{n\rightarrow\infty}f(T_{j}x_{n})
\]

Also we observe that $T$ and $f$ are not reciprocally continuous on $K$; for,
\[
\lim_{n\rightarrow\infty}T_{j}(fx_{n})=\frac{1}{3}+\frac{1}{6j}\neq\frac{2}{3}= T_{j}\left(\frac{2}{3}\right), \text{ and}
\]
\[
\lim_{n\rightarrow\infty}f(T_{j}x_{n})=1\neq\frac{2}{3}=f\left(\frac{2}{3}\right), \text{ for any } j=1,2,3,...
\]

Hence, this example suggests that Theorem \eqref{thrm 4.5} is a generalization of Theorem \eqref{Thrm 1.6}.
\end{example}

\section*{Acknowledgments}
The authors sincerely thank Prof. G.V.R. Babu for critical discussions throughout the preparation of this paper.


\end{document}